\documentclass[12pt]{amsart}%
\usepackage{graphicx}
\usepackage[small,nohug,heads=littlevee]{diagrams}
\usepackage{amscd}
\usepackage{geometry}
\usepackage{amsmath}
\usepackage{amsfonts}
\usepackage{amssymb}%
\providecommand{\U}[1]{\protect\rule{.1in}{.1in}}
\newtheorem{theorem}{Theorem}[section]
\theoremstyle{plain}

\newtheorem{corollary}[theorem]{Corollary}

\newtheorem{lemma}[theorem]{Lemma}

\newtheorem{proposition}[theorem]{Proposition}

\theoremstyle{definition}
\newtheorem{definition}[theorem]{Definition}
\newtheorem{remark}{Remark}
\newtheorem{example}{Example}
\numberwithin{equation}{section}

\begin{document}
\title[Proper homotopy types and $\mathcal{Z}$-boundaries]{Proper homotopy types and $\mathcal{Z}$-boundaries of spaces admitting
geometric group actions}
\author{Craig R. Guilbault}
\address{Department of Mathematical Sciences\\
University of Wisconsin-Milwaukee, Milwaukee, WI 53201}
\email{craigg@uwm.edu}
\author{Molly A. Moran}
\address{Department of Mathematics, The Colorado College, Colorado Springs, Colorado 80903}
\email{molly.moran@coloradocollege.edu}
\thanks{This research was supported in part by Simons Foundation Grants 207264 and
427244, CRG}
\date{January 3, 2018}
\keywords{Absolute neighborhood retract, absolute retract, geometric action, proper
homotopy equivalence, Z-structure, Z-boundary, shape equivalence}

\begin{abstract}
We extend several techniques and theorems from geometric group theory so they
apply to geometric actions on arbitrary proper metric ARs (absolute retracts).
Previous versions often required actions on CW complexes, manifolds, or proper
CAT(0) spaces, or else included a finite-dimensionality hypothesis. We remove
those requirements, providing proofs that simultaneously cover all of the
usual variety of spaces. A second way that we generalize earlier results is by
eliminating a freeness requirement often placed on the group actions. In doing
so, we allow for groups with torsion.

The main theorems are new in that they generalize results found in the
literature, but a significant aim is expository. Toward that end, brief but
reasonably comprehensive introductions to the theories of ANRs (absolute
neighborhood retracts) and $\mathcal{Z}$-sets are included, as well as a much
shorter short introduction to shape theory. Here is a sampling of the theorems
proved here.\medskip

\noindent\textbf{Theorem. }\emph{If quasi-isometric groups }$G$\emph{ and }%
$H$\emph{ act geometrically on proper metric ARs }$X$\emph{ and }$Y$\emph{,
resp., then }$X$\emph{ is proper homotopy equivalent to }$Y$\emph{.\smallskip}

\noindent\textbf{Theorem. }\emph{If quasi-isometric groups }$G$\emph{ and }%
$H$\emph{ act geometrically on proper metric ARs }$X$\emph{ and }$Y$\emph{,
resp., and }$Y$\emph{ can be compactified to a }$\mathcal{Z}$\emph{-structure
}$\left(  \overline{Y},Z\right)  $\emph{ for }$H$\emph{, then the same
boundary can be added to }$X$\emph{ to obtain a }$\mathcal{Z}$\emph{-structure
for }$G$\emph{.\smallskip\smallskip}

\noindent\noindent\textbf{Theorem. }\emph{If quasi-isometric groups }$G$\emph{
and }$H$\emph{ admit }$\mathcal{Z}$\emph{-structures }$\left(  \overline
{X},Z_{1}\right)  $\emph{ and }$\left(  \overline{Y},Z_{2}\right)  $\emph{,
resp., then }$Z_{1}$\emph{ and }$Z_{2}$\emph{ are shape equivalent.}\medskip

\end{abstract}
\maketitle

\section{Introduction\label{Section: Introduction}}

In this paper all spaces are assumed separable and metrizable. A metric space
$\left(  X,d\right)  $ is \emph{proper} if every closed ball in $X$ is
compact. Separability is automatic for proper metric spaces and every proper
metric space is locally compact; conversely, every locally compact space
admits a proper metric.

A locally compact space $X$ is an \emph{absolute neighborhood retract} (ANR)
if, whenever $X$ is embedded as a closed subset of a space $Y$, some
neighborhood of $X$ in $Y$ retracts onto $X$. An ANR $X$ is an \emph{absolute
retract} (AR) if, whenever $X$ is embedded as a closed subset of $Y$, its
image is a retract of $Y$.\footnote{Due to the applications presented here,
local compactness is included as part of the definition of ANR. More general
treatments are commonly found in the literature.}

One of the most significant aspects of \textquotedblleft ANR
theory\textquotedblright\ is that it provides a common ground for studying a
variety of nice spaces: manifolds; locally finite CW (including simplicial and
cube) complexes; proper CAT(0) spaces; Hilbert cube manifolds; etc. This is
particularly useful in subjects where it is desirable to move freely between
categories. In this paper we will generalize several theorems and techniques
from geometric group theory and metric geometry, previously restricted to one
or more of these categories, to the full spectrum of ANRs. A second way that
we improve upon known results is by extending several theorems that previously
applied only to free geometric actions. In our versions, freeness is not
required, so the theorems can be applied to groups with torsion.

For the reader with little or no experience working with abstract ANRs, we
include a short, elementary introduction to ANR theory that is sufficient for
a complete understanding of most of the work presented here. Indeed, a
secondary goal of this paper is to provide the reader that background, and to
provide a level of comfort with this useful category of spaces.\medskip

Here are the most notable theorems to be proved here. Versions of the first
theorem and its corollary are well-known when $X$ and $Y$ are CW complexes
(see \cite[Ch.10]{Geo08}). The traditional proof is inductive over
skeleta---sometimes called a \textquotedblleft connect-the-dots
strategy.\textquotedblright\ Ours relies on a generalized version of that
method, which does not require a CW structure.

\begin{theorem}
\label{Theorem: q-i groups act on phe spaces}If quasi-isometric groups $G$ and
$H$ act geometrically on proper metric ARs $X$ and $Y$, resp., then $X$ is
proper homotopy equivalent to $Y$. In fact, there exist continuous coarse
equivalences $f:X\rightarrow Y$ and $g:Y\rightarrow X$ such that $gf$ and $fg$
are boundedly (hence properly) homotopic to $\operatorname*{id}_{X}$ and
$\operatorname*{id}_{Y}$.
\end{theorem}

The following corollary can be obtained via covering space theory and a deep
result from ANR theory (see Theorem \ref{Theorem: West's Theorem}), when the
$G$-actions are free, and for arbitrary CAT(0) groups by a theorem of Ontaneda
\cite{Ont05}. Our proof is more elementary and the conclusion is more general.

\begin{corollary}
\label{Corollary: Cor to q-i groups act on phe spaces}If a group $G$ acts
geometrically on proper metric ARs $X$ and $Y$, then $X$ is proper homotopy
equivalent to $Y$ via continuous coarse equivalences $f:X\rightarrow Y$ and
$g:Y\rightarrow X$ such that $gf$ and $fg$ are boundedly homotopic to
$\operatorname*{id}_{X}$ and $\operatorname*{id}_{Y}$.
\end{corollary}

The next collection of theorems involves \textquotedblleft$\mathcal{Z}%
$-boundaries\textquotedblright\ of groups---a notion introduced by Bestvina
and expanded upon by Dranishnikov (see
\S \ref{Section: Z-boundaries of groups} for definitions). The idea is to
provide an axiomatic treatment of group boundaries that includes Gromov
boundaries of hyperbolic groups and visual boundaries of CAT(0) groups, but
which can be applied more generally. Both \cite{Bes96} and \cite{Dra06}
recognized ANR theory as the natural setting for such a theory. In order to
obtain some of his most notable conclusions, Bestvina worked only with
finite-dimensional ANRs and torsion-free groups. Dranishnikov relaxed those
conditions, but some of Bestvina's conclusions were then lost. Here we extend
several several theorems from \cite{Bes96} to the more general setting
suggested in \cite{Dra06}.

The first of those theorems allows \textquotedblleft boundary
swapping\textquotedblright\ when a group admits a finite $K\left(  G,1\right)
$ complex. More generally, Bestvina asserted a boundary swapping\ result for
pairs of quasi-isometric groups, each of that type. We obtain generalizations
which allow non-free actions on arbitrary ARs. In addition, we prove an
equivariant version that applies to $E\mathcal{Z}$-boundaries, as defined by
Farrell and Lafont \cite{FaLa05}.

\begin{theorem}
[Boundary Swapping Theorem]\label{Theorem: Boundary Swapping Theorem}Suppose
$G$ acts geometrically on proper metric ARs $X$ and $Y$, and $Y$ can be
compactified to a $\mathcal{Z}$-structure [resp., $E\mathcal{Z}$-structure]
$\left(  \overline{Y},Z\right)  $ for $G$. Then $X$ can be compactified, by
addition of the same boundary, to a $\mathcal{Z}$-structure [resp.,
$E\mathcal{Z}$-structure] $\left(  \overline{X},Z\right)  $ for $G$.
\end{theorem}

For a pair of quasi-isometric groups, the $E\mathcal{Z}$ conclusion no longer
makes sense, but the rest of Theorem \ref{Theorem: Boundary Swapping Theorem}
goes through as follows.

\begin{theorem}
[Generalized Boundary Swapping Theorem]%
\label{Theorem: Generalized Boundary Swapping}Suppose quasi-isometric groups
$G$ and $H$ act geometrically on proper metric ARs $X$ and $Y$, respectively,
and $Y$ can be compactified to a $\mathcal{Z}$-structure $\left(  \overline
{Y},Z\right)  $ for $H$. Then $X$ can be compactified, by addition of the same
boundary, to a $\mathcal{Z}$-structure $\left(  \overline{X},Z\right)  $ for
$G$.
\end{theorem}

\begin{remark}
A \emph{geometric action} is one that is proper, cocompact, and by isometries.
By results to be discussed in \S \ref{Section: Z-boundaries of groups}, it is
enough to assume that the action is proper and cocompact.
\end{remark}

It is well-known that a word hyperbolic group $G$ has a well-defined Gromov
boundary, but a CAT(0) group can admit non-homeomorphic visual boundaries
\cite{CrKl00}. A version of well-definedness can be recovered using shape
theory; one wishes to assert that any two $\mathcal{Z}$-boundaries of a group
$G$ are shape equivalent. That assertion---implicit in \cite{Geo86}---was made
explicitly in \cite{Bes96}, for certain torsion-free groups. Later, Ontaneda
\cite{Ont05} proved an analogous theorem, without a torsion-free hypothesis,
for the special case of visual boundaries of CAT(0) groups. Here we show that
all of these additional hypotheses are unnecessary.

\begin{theorem}
\label{Theorem: shape of bdry a quasi-isom invariant}If quasi-isometric groups
$G$ and $H$ admit $\mathcal{Z}$-structures $\left(  \overline{X},Z_{1}\right)
$ and $\left(  \overline{Y},Z_{2}\right)  $, respectively, then $Z_{1}$ and
$Z_{2}$ are shape equivalent.
\end{theorem}

\begin{corollary}
\label{Corollary: Shape of bdry a group invariant}If a group $G$ admits a
$\mathcal{Z}$-boundary, then that boundary is well-defined up to shape equivalence.
\end{corollary}

\section{Some basics of ANR theory\label{Section: Basics of ANRs}}

In this section we cover the necessary background from ANR theory. Rather than
simply quoting results from the literature, we provide a brief, elementary
treatment of the topic---including proofs of most key facts used in this
paper. More complete treatments can be found in \cite{Hu65} or \cite{vMi89}.

Recall that a locally compact space $X$ is an an \emph{absolute neighborhood
retract} (ANR) if, whenever $X$ is embedded as a closed subset of another
space $Y$, some neighborhood of $X$ retracts onto $X$; if the entire space $Y$
always retracts onto $X$, we call $X$ an\ \emph{absolute retract} (AR).

\begin{remark}
A finite-dimensional ANR [resp., AR] is often called a \emph{Euclidean
neighborhood retract} (ENR) [resp. \emph{Euclidean retract} (ER)]. A key
aspect of this paper is the development of techniques that do not require a
restriction to this subcategory.
\end{remark}

Some of the most important properties of A[N]Rs involve extensions of maps. In
fact, an alternative approach \emph{defines }A[N]R using these very extension
properties. In that setting, they are sometimes called A[N]Es (absolute
[neighborhood] extensors). We will not use that terminology, but the following
characterization is crucial.

\begin{lemma}
[Extensor characterization of ARs and ANRs ]%
\label{Lemma: ANE characterization}A locally compact space $X$ is an AR
[resp., ANR] if and only if it satisfies the following extension
property:\medskip

\noindent(\dag) If $A$ is a closed subset of an arbitrary space $Y$ and
$f:A\rightarrow X$ is continuous, then there is a continuous extension
$\overline{f}:Y\rightarrow X$ [resp., a continuous extension $f:U\rightarrow
X$, where $U$ is a neighborhood of $A$ in $Y$.].
\end{lemma}

\begin{proof}
We begin with the reverse implications. Suppose $X$ is embedded as a closed
subset of a space $Y$ and consider the identity map $\operatorname*{id}%
_{X}:X\rightarrow X$. By viewing the domain copy of $X$ as a closed subset of
$Y$, the absolute extension property guarantees a map $\overline
{f}:Y\rightarrow X$ that restricts to the identity on $X$, i.e., $\overline
{f}$ is a retraction. If we assume the weaker extension property, we get a
retraction $\overline{f}:U\rightarrow X$.

For the forward implications, choose a proper metric $d$ on $X$, let $\left\{
x_{i}\right\}  _{i=1}^{\infty}\subseteq X$ be a dense subset, and for each $i$
define $g_{i}\left(  x\right)  =d\left(  x_{i},x\right)  $. Then
$\mathbf{g=}\left(  g_{i}\right)  $ embeds $X$ as a closed subset of $%
\mathbb{R}
^{\infty}$ (with the product topology). In what follows, view $X$ as a closed
subset of $%
\mathbb{R}
^{\infty}$ and let $j:X\hookrightarrow%
\mathbb{R}
^{\infty}$ be the inclusion map.

Suppose $A$ is a closed subset of a space $Y$ and $f:A\rightarrow X$ is
continuous. By applying the Tietze Extension Theorem coordinate-wise, extend
$jf:A\rightarrow%
\mathbb{R}
^{\infty}$ to a continuous map $F:Y\rightarrow%
\mathbb{R}
^{\infty}$. If $X$ is an AR, choose a retraction $r:%
\mathbb{R}
^{\infty}\rightarrow X$ and let $\overline{f}=rF$. If $X$ is an ANR, choose a
retraction $r:V\rightarrow X$, where $V$ is a neighborhood of $X$ in $%
\mathbb{R}
^{\infty}$; then let $U=F^{-1}\left(  V\right)  $ and $\overline{f}=r\left.
F\right\vert _{U}$.
\end{proof}

\begin{proposition}
[Homotopy Extension Property]\label{Proposition: HEP}Let $A$ be a closed
subset of a space $Y$, $T=\left(  Y\times\left\{  0\right\}  \right)  \cup
A\times\left[  0,1\right]  $ and $h:T\rightarrow X$ a map into an ANR. Then
there is an extension of $h$ to a homotopy $H:Y\times\left[  0,1\right]
\rightarrow X$. If the homotopy $\left.  h\right\vert _{A\times\left[
0,1\right]  }$ is $K$-bounded\footnote{Bounded homotopies are defined in
\S \ref{Section: Group theory and metric geometry}.} and $Y$ is locally
compact, then $H$ can be chosen to be $\left(  K+1\right)  $-bounded.
\end{proposition}

\begin{proof}
By Lemma \ref{Lemma: ANE characterization} there is an extension of $h$ to
$\overline{h}:U\rightarrow X$, where $U$ is a neighborhood of $T$ in
$Y\times\left[  0,1\right]  $. Let $V\subseteq Y$ be an open neighborhood of
$A$ such that $V\times\left[  0,1\right]  \subseteq U$, and choose a Urysohn
function $\lambda:Y\rightarrow\left[  0,1\right]  $ with $\lambda\left(
Y\backslash V\right)  =0$ and $\lambda\left(  A\right)  =1$. Define
\[
\overline{H}\left(  y,t\right)  =\left\{
\begin{array}
[c]{ll}%
\overline{h}\left(  y,\lambda\left(  y\right)  \cdot t\right)  & \text{if
}y\in V\\
h\left(  y,0\right)  & \text{if }y\notin V
\end{array}
\right.
\]
Now suppose $\left.  h\right\vert _{A\times\left[  0,1\right]  }$ is
$K$-bounded and $Y$ is locally compact. Then each $a\in A$ has a compact
neighborhood $N_{a}\subseteq V$; so by uniform continuity, there is an open
neighborhood $V_{a}\subseteq N_{a}$ over which all tracks of $\left.
\overline{h}\right\vert _{V_{a}\times\left[  0,1\right]  }$ have diameter
$<K+1$. Rechoose $V$ to be $\cup_{a\in A}V_{a}$; then $\overline{H}$, as
defined above, is $(K+1)$-bounded.
\end{proof}

\begin{corollary}
\label{Corollary: null-homotopic extension}If $A$ is a closed subset of an
arbitrary space $Y$ and $f:A\rightarrow X$ is a null-homotopic map into an
ANR, then $f$ extends to a map $\overline{f}:Y\rightarrow X$.
\end{corollary}

\begin{proof}
Let $J:A\times\left[  0,1\right]  \rightarrow X$ be a homotopy with
$J_{0}\left(  A\right)  =\left\{  p\right\}  $ and $J_{1}=f$. Extend $J$ to a
map of $h:\left(  Y\times\left\{  0\right\}  \right)  \cup A\times\left[
0,1\right]  \rightarrow X$ by sending $Y\times\left\{  0\right\}  $ to $p$;
then apply the Homotopy Extension Property and let $\overline{f}=H_{1}$.
\end{proof}

\begin{corollary}
\label{Corollary: AR iff contractible ANR}An ANR is an AR if and only if it is contractible.
\end{corollary}

\begin{proof}
Suppose $X$ is an $AR$. By the argument used in Lemma
\ref{Lemma: ANE characterization}, we may assume $X$ is a closed subset of $%
\mathbb{R}
^{\infty}$; so by hypothesis, there is a retraction $r:%
\mathbb{R}
^{\infty}\rightarrow X$. Composing any contraction of $%
\mathbb{R}
^{\infty}$ with $r$ gives a contraction of $X$.

For the converse, assume $X$ is contractible and $X\hookrightarrow Y$ is an
embedding as a closed subset of a space $Y$. Then $\operatorname*{id}%
_{X}:X\rightarrow X$ is a null-homotopic map into an ANR, so by Corollary
\ref{Corollary: null-homotopic extension} it extends to a map of $Y$ into $X$.
That map is a retraction.
\end{proof}

\begin{corollary}
\label{Corollary: open subsets of ANRs are ANRs}Every open subset of an ANR is
an ANR.
\end{corollary}

\begin{proof}
Let $V$ be an open subset of an ANR $X$, and $f:A\rightarrow V$ a continuous
map, where $A$ is a closed subset of a space $Y$. By Lemma
\ref{Lemma: ANE characterization}, there is an open neighborhood $U^{\prime}$
of $A$ in $Y$ and an extension $f^{\prime}:U^{\prime}\rightarrow X$ of $f$.
Let $U=U^{\prime}\cap(f^{\prime})^{-1}\left(  V\right)  $ and $\overline
{f}=\left.  f^{\prime}\right\vert _{U}$.
\end{proof}

With additional work, one can prove a similar, but different proposition.

\begin{proposition}
A space $X$ with the property that each $x\in X$ has an open neighborhood that
is an ANR, is itself an ANR.
\end{proposition}

The next proposition provides a wealth of examples.

\begin{proposition}
\label{Proposition: lots of A[N]Rs}Being an AR [ANR] is a topological
property. Furthermore,

\begin{enumerate}
\item \label{Assertion 1: Lots of ANRs}$\left[  0,1\right]  ,[0,1),$ and
$\left(  0,1\right)  $ are ARs,

\item \label{Assertion 2: Lots of ANRs}every finite product of A[N]Rs is an A[N]R,

\item \label{Assertion3: Lots of ANRs}a countably infinite product of ARs is
an AR, provided all but finitely many factors are compact,

\item \label{Assertion 4: Lots of ANRs}Every retract of an A[N]R is an A[N]R.
\end{enumerate}
\end{proposition}

\begin{proof}
That these properties are invariant under homeomorphism is clear. Assertion
(\ref{Assertion 1: Lots of ANRs}) follows from Lemma
\ref{Lemma: ANE characterization}, Corollaries
\ref{Corollary: AR iff contractible ANR} and
\ref{Corollary: open subsets of ANRs are ANRs}, and the Tietze Extension Theorem.

For assertion (\ref{Assertion 2: Lots of ANRs}), let $A\subseteq Y$ be closed
and $\mathbf{f=}\left(  f_{i}\right)  :A\rightarrow\prod\nolimits_{i=1}%
^{k}X_{i}$ be continuous. If each $X_{i}$ is an AR, choose extensions
$\overline{f}_{i}:Y\rightarrow X_{i}$ to get an extension $\overline
{\mathbf{f}}=\left(  \overline{f}_{i}\right)  :Y\rightarrow\prod
\nolimits_{i=1}^{k}X_{i}$. If each $X_{i}$ is an ANR, choose extensions
$\overline{f}_{i}:U_{i}\rightarrow X_{i}$, where $U_{i}$ a neighborhood of $A$
in $Y$, then let $U=\cap U_{i}$ and $\overline{\mathbf{f}}=\left(  \left.
\overline{f}_{i}\right\vert _{U}\right)  :U\rightarrow\prod\nolimits_{i=1}%
^{k}X_{i}$.

For infinite products of ARs, we must restrict to countable products to ensure
metrizability; and to ensure local compactness, only finitely many factors can
be noncompact. With those caveats, the above proof remains valid for ARs (but
fails for ANRs).

To prove assertion (\ref{Assertion 4: Lots of ANRs}), first recall that if
$r:X\rightarrow X_{0}$ is a retraction, then $X_{0}$ is closed in $X$, so
$X_{0}$ is locally compact. Now suppose $f:A\rightarrow X_{0}$ is continuous,
with $A$ a closed subset of a space $Y$. If $X$ is an AR, there is an
extension $F:Y\rightarrow X$. Then $\overline{f}=rF:Y\rightarrow X_{0}$ is an
extension, showing that $X_{0}$ is an AR. By the same approach, if $X$ is an
ANR, then so is $X_{0}$.
\end{proof}

Recall that a space $X$ is \emph{locally contractible} \emph{at }$x\in X$ if
every neighborhood $U$ of $x$ contains a neighborhood $V$, such that $V$
contracts in $U$. A space that is locally contractible at each of its points
is \emph{locally contractible}.

\begin{proposition}
Every ANR is locally contractible and every finite-dimensional locally compact
and locally contractible space is an ANR.
\end{proposition}

\begin{proof}
[Sketch of proof]It is easy to see that a retract of a locally contractible
space is locally contractible. If $X$ is an ANR then there is an embedding
$X\hookrightarrow%
\mathbb{R}
^{\infty}$ as a closed subset and a retraction $r:U\rightarrow X$ of an open
neighborhood onto $X$. Since $U$ is locally contractible, so is $X$.

If $X$ is locally compact and finite-dimensional, there is a proper closed
embedding of $X\hookrightarrow%
\mathbb{R}
^{n}$ for some $n<\infty$. By Corollary
\ref{Corollary: open subsets of ANRs are ANRs} and Proposition
\ref{Proposition: lots of A[N]Rs}, it suffices to exhibit a retraction of some
neighborhood onto $X$. Choose a polyhedral neighborhood $N_{0}$ of $X$ and an
infinite triangulation of $N_{0}-X$ that gets progressively finer near $X$.
Define $r_{0}:N_{0}^{\left(  0\right)  }\cup X\rightarrow X$ by sending each
vertex of $N-X$ to a nearest point in $X$. Assume inductively that there is a
polyhedral subneighborhood $N_{k}$ of $X$ and a retraction $r_{k}%
:N_{k}^{\left(  k\right)  }\cup X\rightarrow X$, where $N_{k}^{\left(
k\right)  }$ is the $k$-skeleton of $N_{k}-X$. By the local contractibility of
$X$, there exists a polyhedral neighborhood $N_{k+1}\subseteq N_{k}$,
sufficiently small, that for each $\left(  k+1\right)  $-simplex $\sigma
^{k+1}$ of $N_{k+1}$, $\left.  r_{k}\right\vert _{\partial\sigma^{k}}$ extends
over $\sigma^{k+1}$ to a map into $X$, thereby giving a retraction
$r_{k+1}:N_{k+1}^{\left(  k+1\right)  }\cup X\rightarrow X$. The desired
retraction follows by induction.
\end{proof}

\begin{example}
Using the above observations, one sees that: every finite-dimensional or
Hilbert cube manifold; every locally finite CW, simplicial, or cube complex;
and every finite-dimensional proper CAT(0) space is an ANR. (In fact,
\cite{Ont05} shows that \emph{every} proper CAT(0) space is an ANR, hence an AR.)
\end{example}

A few deeper facts about ANRs will play a role in this paper. We state them
here without proofs.

\begin{theorem}
[Hanner's Theorem \cite{Han51}]\label{Theorem: Hanner's Theorem}A space $X$ is
an ANR if for every open cover $\mathcal{U}$ of $X$ there is an ANR $Y$ and
maps $f:X\rightarrow Y$ and $g:Y\rightarrow X$ such that $gf$ is $\mathcal{U}%
$-homotopic to $\operatorname*{id}_{X}$.
\end{theorem}

\begin{theorem}
[West's Theorem \cite{Wes77}]\label{Theorem: West's Theorem}Every ANR is
proper homotopy equivalent to a locally finite polyhedron; every compact ANR
is homotopy equivalent to a finite polyhedron.
\end{theorem}

\begin{theorem}
[Edwards' Theorem \cite{Edw80}]\label{Theorem: Edwards' Theorem}If $X$ is a
locally compact ANR, then $X\times\left[  0,1\right]  ^{\infty}$ is a Hilbert
cube manifold.
\end{theorem}

\section{$\mathcal{Z}$-sets and $\mathcal{Z}$%
-com\-pact\-i\-fi\-ca\-tions\label{Section: Z-sets}}

A closed subset $A$ of a space $X$, is a \emph{$\mathcal{Z}$-set} if there
exists a homotopy $H:X\times\lbrack0,1]\rightarrow X$ such that $H_{0}%
=\operatorname*{id}_{X}$ and $H_{t}(X)\subset X-A$ for every $t>0$. In this
case we say that $H$ \emph{instantly homotopes }$X$ \emph{off from} $A$.

\begin{remark}
Notice that, if $A\subseteq X$ is a $\mathcal{Z}$-set, then $A$ is nowhere
dense in $X$.
\end{remark}

\begin{example}
The prototypical $\mathcal{Z}$-set is the boundary of a manifold. More
generally, a closed subset $A$ of an $n$-manifold $M^{n}$ is a $\mathcal{Z}%
$-set if and only if $A\subseteq\partial M^{n}$.
\end{example}

The following lemma generalizes \cite[Prop.1.6]{Fer00}. It will play a crucial
role in our boundary swapping theorems.

\begin{lemma}
\label{Lemma: Z-set criterion}Let $f:\left(  X,A\right)  \rightarrow\left(
Y,B\right)  $ and $g:\left(  Y,B\right)  \rightarrow\left(  X,A\right)  $ be
continuous maps with $f\left(  X-A\right)  \subseteq Y-B$, $g\left(
Y-B\right)  \subseteq X-A$, and $\left.  gf\right\vert _{A}=\operatorname*{id}%
_{A}$. Suppose further that there is a homotopy $J:X\times\left[  0,1\right]
\rightarrow X$ which is fixed on $A$ and satisfies: $J_{0}=\operatorname*{id}%
_{X}$, $J_{1}=gf$, and $J\left(  (X-A)\times\left[  0,1\right]  \right)
\subseteq X-A$. If $B$ is a $\mathcal{Z}$-set in $Y$, then $A$ is a
$\mathcal{Z}$-set in $X$.
\end{lemma}

\begin{proof}
Since $A=g^{-1}\left(  B\right)  $, $A$ is closed in $X$. Choose
$K:Y\times\left[  0,1\right]  \rightarrow Y$ that instantly homotopes $Y$ off
from $B$. We will construct $H:X\times\left[  0,1\right]  \rightarrow X$,
which instantly homotopes $X$ off from $A$, by describing the track of each
$x\in X$.

For each $x\in X$ define $\alpha_{x}:\left[  0,1\right]  \rightarrow X$ by
$\alpha_{x}\left(  t\right)  =J_{t}\left(  x\right)  $ and $\beta_{x}:\left[
0,1\right]  \rightarrow X$ by $\beta_{x}\left(  t\right)  =gK_{t}\left(
f\left(  x\right)  \right)  $. Note that

\begin{enumerate}
\item $\alpha_{x}\left(  0\right)  =x$ and $\alpha_{x}\left(  1\right)
=gf\left(  x\right)  =\beta_{x}\left(  0\right)  $,

\item \label{Observation about points of X-A}if $x\in X-A$ then $\alpha_{x}$,
$\beta_{x}\subseteq X-A$,

\item \label{Observation about points of A}if $x\in A$ then $\alpha_{x}\equiv
x$, $\beta_{x}\left(  0\right)  =x$, and $\beta_{x}\left(  (0,1]\right)
\subseteq X-A$, and

\item \label{Observation on diam of alpha}$\operatorname*{diam}\left(
\alpha_{x}\right)  \rightarrow0$ as $d\left(  x,A\right)  \rightarrow0$.
\end{enumerate}

The track $\gamma_{x}$ of $x$ under $H$ will follow the concatenation
$\alpha_{x}\cdot\beta_{x}$, but in order for points to be instantly homotoped
off from $A$, reparameterizations are necessary. Let $r_{x}=\min\left\{
d\left(  x,A\right)  ,\frac{1}{2}\right\}  $ and define $\gamma_{x}:\left[
0,1\right]  \rightarrow X$ as follows:%

\[
\gamma_{x}\left(  t\right)  =\left\{
\begin{tabular}
[c]{ll}%
$\alpha_{x}\left(  t/r_{x}\right)  $ & if $0\leq t<r_{x}$\\
$\beta_{x}\left(  t-r_{x}/1-r_{x}\right)  $ & if $r_{x}\leq t\leq1$%
\end{tabular}
\ \right.
\]

\noindent Observation \ref{Observation on diam of alpha} and the fact that
$\alpha_{x}$, $\beta_{x}$, and $r_{x}$ vary continuously with $x$ combine to
show that $\gamma_{x}$ varies continuously with $x$. Define $H\left(
x,t\right)  =\gamma_{x}\left(  t\right)  $ and apply Observations
\ref{Observation about points of X-A} and \ref{Observation about points of A}
to deduce that $H$ instantly homotopes $X$ off from $A$.
\end{proof}

A \emph{$\mathcal{Z}$-com\-pact\-i\-fi\-ca\-tion} of a space $Y$ is a
com\-pact\-i\-fi\-ca\-tion $\overline{Y}$ such that $Z\equiv\overline{Y}-Y$ is
a $\mathcal{Z}$-set in $\overline{Y}$. In that case we call $Z$ a
\emph{$\mathcal{Z}$-boundary} for $Y$. Here we follow standard convention for
com\-pact\-i\-fi\-ca\-tions (\cite{Mun00} or \cite{Dug78}) by requiring
$\overline{Y}$ to be Hausdorff and $Y$ to be dense in $\overline{Y}$. Under
that convention: $Y$ must be locally compact and Hausdorff; $Y$ is necessarily
open in $\overline{Y}$; and $Z$ is compact and nowhere dense.

Lemmas \ref{Lemma: metrizability of Z-compactifications} and
\ref{Lemma: Z-compactifications are ANRs} allow us to stay within preferred
categories of spaces when taking $\mathcal{Z}$-com\-pact\-i\-fi\-ca\-tions.

\begin{lemma}
\label{Lemma: metrizability of Z-compactifications}If $\overline{Y}$ is a
\emph{$\mathcal{Z}$}-com\-pact\-i\-fi\-ca\-tion of a (separable metrizable)
space $Y$, then $\overline{Y}$ is also separable and metrizable.
\end{lemma}

\begin{proof}
For metrizability, we apply the Urysohn Metrization Theorem \cite[p.214]%
{Mun00}. Since $\overline{Y}$ is compact and Hausdorff it is a normal space,
so we need only check that there is a countable basis for its topology. Being
separable and metrizable, $Y$ admits a countable basis $\mathcal{B}_{0}$. For
each integer $i>0$, let $\mathcal{B}_{i}=\left\{  H_{1/i}^{-1}\left(
B\right)  \mid B\in\mathcal{B}_{0}\right\}  $ and note that $\overline
{\mathcal{B}}=\cup_{i\geq0}\mathcal{B}_{i}$ is a countable basis for
$\overline{Y}$.

Since $\overline{Y}$ is compact and metrizable, it is separable.
\end{proof}

\begin{lemma}
\label{Lemma: Z-compactifications are ANRs}If $\overline{Y}$ is a
\emph{$\mathcal{Z}$}-com\-pact\-i\-fi\-ca\-tion of an ANR $Y$, then
$\overline{Y}$ is an ANR. If $Y$ is an AR, then so is $\overline{Y}$.
\end{lemma}

\begin{proof}
The fact that $\overline{Y}$ is an ANR is a straight forward consequence of
Theorem \ref{Theorem: Hanner's Theorem}. For the latter observation, use the
definition of $\mathcal{Z}$-set to homotope $\overline{Y}$ into $Y$, then
follow that homotopy with a contraction of $Y$ to obtain a contraction of
$\overline{Y}$. Now apply Corollary \ref{Corollary: AR iff contractible ANR}.
\end{proof}

\begin{lemma}
\label{Lemma: Z-sets in subspaces}Let $N$ be a neighborhood of a
\emph{$\mathcal{Z}$}-set $A$ in an ANR $X$. Then $A$ is a \emph{$\mathcal{Z}$%
}-set in $N$.
\end{lemma}

\begin{proof}
Let $H:X\times\lbrack0,1]\rightarrow X$ instantly homotope $X$ off from $A$,
and choose open sets $U,U^{\prime}\subseteq X$ such that%
\[
A\subseteq U\subseteq\overline{U}\subseteq U^{\prime}\subseteq\overline
{U^{\prime}}\subseteq\operatorname*{int}N\text{.}%
\]
By truncating $H$ at a time $0<t_{0}\leq1$ then reparameterizing, we may
obtain a homotopy $K:\overline{U}\times\left[  0,1\right]  \rightarrow
\operatorname*{int}N$ such that $K\left(  \overline{U}\times\left[
0,1\right]  \right)  \subseteq U^{\prime}$. Extend $K$ to the trivial
(constantly identity) homotopy on $(N-U^{\prime})\times\left[  0,1\right]  $
and the identity map on $N\times\left\{  0\right\}  $. By Corollary
\ref{Corollary: open subsets of ANRs are ANRs}, we can apply the Homotopy
Extension Property to obtain $K:\operatorname*{int}N\times\left[  0,1\right]
\rightarrow\operatorname*{int}N,$ after which we can extend to the identity
over the frontier of $N$. Moreover, by choosing a sufficiently small
neighborhood $V$ of $(\operatorname*{int}N-U^{\prime})\cup\overline{U}$ (as
was used in the proof of the Homotopy Extension Property) we can arrange that
no track of $K$ passes through $A$.
\end{proof}

\begin{remark}
\label{Remark: equivalent formulation of Z-set}If one restricts attention to
ANRs, Lemmas \ref{Lemma: Z-compactifications are ANRs} and
\ref{Lemma: Z-sets in subspaces} allow a variety of equivalent formulations of
$\mathcal{Z}$-set and $\mathcal{Z}$-com\-pact\-i\-fi\-ca\-tion. For example, a
closed subset $A$ of an ANR $X$ is a $\mathcal{Z}$-set if and only if
$U-A\hookrightarrow U$ is a homotopy equivalence for every open set $U$ in
$Y$. See \cite{Hen75}.
\end{remark}

The next definition blends topology and geometry in that the specific metric
on $X$ plays a role.

\begin{definition}
\label{Definition: controlled Z-compactification}A \emph{controlled
}$\mathcal{Z}$\emph{-com\-pact\-i\-fi\-ca\-tion }of a proper metric space
$\left(  Y,d\right)  $ is a $\mathcal{Z}$-com\-pact\-i\-fi\-ca\-tion
$\overline{Y}$ satisfying the additional condition:\medskip

\noindent\emph{(\ddag)} For every $R>0$ and every open cover $\mathcal{U}$ of
$\overline{Y}$, there is a compact set $C\subset Y$ so that if $A\subseteq
Y-C$ and $\operatorname*{diam}_{d}A<R$, then $A\subseteq U$ for some
$U\in\mathcal{U}$.
\end{definition}

\begin{example}
The (standard) com\-pact\-i\-fi\-ca\-tion of hyperbolic $n$-space, by addition
of an end point to each ray emanating from the origin, is a controlled
$\mathcal{Z}$-com\-pact\-i\-fi\-ca\-tion; so is the analogous
com\-pact\-i\-fi\-ca\-tion of Euclidean $n$-space. More generally, adding the
visual sphere at infinity, with the cone topology, to a proper CAT(0) space is
a controlled $\mathcal{Z}$-com\-pact\-i\-fi\-ca\-tion.

The com\-pact\-i\-fi\-ca\-tion $\overline{%
\mathbb{R}
}=%
\mathbb{R}
\cup\left\{  \pm\infty\right\}  $ is a controlled $\mathcal{Z}$%
-com\-pact\-i\-fi\-ca\-tion of $%
\mathbb{R}
$. By contrast, $\overline{%
\mathbb{R}
}\times\overline{%
\mathbb{R}
}$ is a $\mathcal{Z}$-com\-pact\-i\-fi\-ca\-tion of $%
\mathbb{R}
^{2}$, but not a controlled $\mathcal{Z}$-com\-pact\-i\-fi\-ca\-tion (under
the Euclidean metric).
\end{example}

\begin{remark}
By Lemma \ref{Lemma: metrizability of Z-compactifications}, we may place a
metric $\overline{d}$ on $\overline{Y}$. Then, using Lebesgue numbers, (\ddag)
can be reformulated as: \medskip

\noindent(\ddag$^{\prime}$) For every $R>0$ and $\varepsilon>0$, there is a
compact $C\subset Y$ so that if $A\subseteq Y-C$ and $\operatorname*{diam}%
_{d}A<R$, then $\operatorname*{diam}_{\overline{d}}A<\varepsilon$.\medskip

Except for an impact on the size of $C$, the specific $\overline{d}$ chosen is
unimportant; moreover, there is no direct relationship between $d$ and
$\overline{d}$. The following useful lemma highlights the difference between
those metrics.
\end{remark}

\begin{lemma}
\label{Lemma: local property of controlled Z-compactifications}Suppose
$\overline{Y}=Y\cup Z$ is a controlled $\mathcal{Z}$%
-com\-pact\-i\-fi\-ca\-tion of $Y$. For each $z\in Z$, each neighborhood
$\overline{U}$ of $z$ in $\overline{Y}$, and each $R>0$, there is a
neighborhood $\overline{V}$ of $z$ so that $d_{Y}(V,Y-U)\geq R$.

\begin{proof}
Place a metric $\overline{d}$ on $\overline{Y}$, then choose a neighborhood
$\overline{U^{\prime}}$ of $z$ whose closure is contained in the interior of
$\overline{U}$ and let $\epsilon=\overline{d}(\overline{U^{\prime}}%
,\overline{Y}-\overline{U})$. From the control condition there is a
$C\subseteq Y$ so that sets in $Y-C$ of $d$-diameter less than $R$ have
$\overline{d}$-diameter less than $\epsilon/2$. Set $\overline{V}%
=\overline{U^{\prime}}-N_{R}\left(  C\right)  $, where $N_{R}\left(  C\right)
$ is the closed $R$-neighborhood of $C$, and suppose there exist $x\in V$ and
$y\in Y-U$ with $d(x,y)<R$. Then $y\in Y-C$, so $\{x,y\}\subseteq Y-C$ and
$\operatorname*{diam}_{d}\{x,y\}<R$. By the control condition $\overline
{d}(x,y)<\epsilon/2$, a contradiction.
\end{proof}
\end{lemma}

\section{Preliminaries from geometric group theory and metric
geometry\label{Section: Group theory and metric geometry}}

Through the remainder of this paper, functions (also called maps) are not
always continuous. When continuity is assumed or required, it will be done
explicitly. Since the concepts presented here are geometric, all spaces are
assumed to come with a fixed metric. We use $B_{d}\left(  x,r\right)  $ and
$B_{d}\left[  x,r\right]  $ to denote open and closed metric balls, respectively.

\begin{definition}
Let $X$ be a space, $\mathcal{A}$ a collection of subsets of $X$, and
$D\subseteq X$.

\begin{enumerate}
\item $\mathcal{A}$ is \emph{locally discrete }if each $x\in X$ has a
neighborhood intersecting at most one element of $\mathcal{A}$

\item $\mathcal{A}$ is \emph{uniformly bounded }if the set $\left\{
\operatorname*{diam}\left(  A\right)  \mid A\in\mathcal{A}\right\}  $ is
bounded above, and

\item $D$ is \emph{large-scale dense} (also called \emph{quasi-dense}) if
$\left\{  d\left(  x,D\right)  \mid x\in X\right\}  $ is bounded above.
\end{enumerate}
\end{definition}

\begin{definition}
Let $f,g:X\rightarrow Y$ be functions. Then

\begin{enumerate}
\item $f$ is \emph{large-scale surjective} if $f\left(  X\right)  $ is
large-scale dense in $Y$,

\item $f$ is \emph{metrically proper} if the pre-image of every bounded subset
of $Y$ is a bounded in $X$,

\item $f$ is \emph{large-scale uniform} (also called bornologous) if, for
every $R>0$, there is an $S>0$ so that if $d_{X}(x,x^{\prime})<R$, then
$d_{Y}(f(x),f(x^{\prime}))<S$,

\item $f$ and $g$ are \emph{boundedly close} if $\left\{  \operatorname*{diam}%
\left(  f\left(  x\right)  ,g\left(  x\right)  \right)  \mid x\in X\right\}  $
is bounded above.
\end{enumerate}
\end{definition}

We reserve the term \emph{proper} for its traditional meaning: a continuous
function $f:X\rightarrow Y$ for which $f^{-1}(C)$ is compact whenever $C$ is
compact. Similarly, \emph{homotopy} always indicates a continuous function. A
\emph{bounded homotopy }$H:X\times\left[  0,1\right]  \rightarrow Y$ is one
for which the collection $\left\{  H\left(  \{x\}\times\left[  0,1\right]
\right)  \mid x\in X\right\}  $ is uniformly bounded. A continuous function
$f:X\rightarrow Y$ is a \emph{proper homotopy equivalence} if there exists a
continuous $g:Y\rightarrow X$ such that $gf$ and $fg$ are homotopic, via
proper homotopies, to $\operatorname*{id}_{X}$ and $\operatorname*{id}_{Y}$,
respectively. The following is immediate.

\begin{lemma}
Let $f,g:X\rightarrow Y$ be functions between proper metric spaces. Then

\begin{enumerate}
\item if $f$ and $g$ are boundedly close and $f$ is metrically proper, then so
is $g$,

\item if $f$ and $g$ are large-scale uniform and boundedly close over a
large-scale dense subset of $X$, then $f$ and $g$ are boundedly close,

\item $f$ is metrically proper and continuous if and only if $f$ is proper, and

\item if $H:X\times\left[  0,1\right]  \rightarrow Y$ is a bounded homotopy,
then $H$ is proper if and only if $H_{t}$ is proper for some [resp., all] $t$.
\end{enumerate}
\end{lemma}

The next set of definitions provides useful generalizations of
\textquotedblleft quasi-isometric embedding\textquotedblright\ and
\textquotedblleft quasi-isometry\textquotedblright.

\begin{definition}
A map $f:X\rightarrow Y$ is:

\begin{enumerate}
\item \emph{coarse} if it is metrically proper and large-scale uniform;

\item a \emph{coarse equivalence} if it has a coarse inverse, i.e., a coarse
map $g:Y\rightarrow X$ such that $gf$ and $fg$ are boundedly close to
$\operatorname*{id}_{X}$ and $\operatorname*{id}_{Y}$;

\item a \emph{coarse embedding} if $f:X\rightarrow f\left(  X\right)  $ is a
coarse equivalence.
\end{enumerate}
\end{definition}

\begin{remark}
An equivalent formulation of coarse embedding is the existence of
non-decreasing, metrically proper functions $\rho_{-},\rho_{+}:[0,\infty
)\rightarrow\lbrack0,\infty)$ such that
\[
\rho_{-}\left(  d\left(  x,x^{\prime}\right)  \right)  \leq d\left(  f\left(
x\right)  ,f\left(  x^{\prime}\right)  \right)  \leq\rho_{+}\left(  d\left(
x,x^{\prime}\right)  \right)
\]
for all $x,x^{\prime}\in X$, with $f$ being a coarse equivalence if it is also
large-scale surjective. Quasi-isometric embeddings and quasi-isometries are
the special cases where $\rho_{-}$ and $\rho_{+}$ can be chosen to be affine
linear functions.
\end{remark}

A group action on a metric space $X$ is \emph{geometric} if it is proper,
cocompact, and by isometries. Here \emph{cocompact }means that there exists a
compact $K\subseteq X$ such that $GK=X$, and \emph{proper} (sometimes called
\emph{properly discontinuous}) means that, for any compact $K\subseteq X$, the
set $\left\{  g\in G\mid gK\cap K\neq\varnothing\right\}  $ is finite. A
useful application of the notion of coarse equivalence is the following
variation on the classical \v{S}varc-Milnor Lemma.

\begin{proposition}
[Generalized \v{S}varc-Milnor]\label{Proposition: Generalized Svarc-Milnor}%
Suppose $G$ acts geometrically on a connected proper metric space $X$. Then
$G$ is finitely generated and, when $G$ is endowed with a corresponding word
metric and $x_{0}\in X$, the map $g\mapsto gx_{0}$ is a coarse equivalence.
\end{proposition}

\begin{proof}
This is a special case of \cite[Cor.0.9]{BDM07}. Our version is simpler since
the finite generation of $G$ (which is standard \cite[Th.I.8.10]{BrHa99}),
allows the use of a word metric, which (also standard) is well-defined up to quasi-isometry.
\end{proof}

\begin{remark}
For length spaces (hence, for finitely generated groups) coarsely equivalent
and quasi-isometric are equivalent notions \cite[p.19]{NoYu12}. A nice aspect
of Proposition \ref{Proposition: Generalized Svarc-Milnor} is that $X$ need
not be a length space.
\end{remark}

The \emph{order }of an open cover $\mathcal{U}$ of a space $X$ is the largest
integer $k$ such that some $x\in X$ is contained in $k$ members of
$\mathcal{U}$. Classical \emph{Lebesgue covering dimension} of $X$ looks at
orders of open covers with arbitrarily small mesh; at the other extreme,
asymptotic dimension of $X$ considers orders of uniformly bounded covers with
arbitrarily large Lebesgue numbers. The following definition is far less rigid
than either of these, requiring a single uniformly bounded open cover of a
given index.

\begin{definition}
A space $X$ has \emph{finite macroscopic dimension }if it admits a uniformly
bounded open cover of finite order. If the order of that cover is $n+1$, we
write $\operatorname*{mdim}X\leq n$; if $n$ is the minimum such integer, we
say $\operatorname*{mdim}X=n$.
\end{definition}

\begin{definition}
A space $X$ is \emph{uniformly contractible} if for each $R>0$, there exists
$S>R$ so that every ball $B(x,R)$ contracts in $B(x,S)$.
\end{definition}

The most fundamental examples of the above two definitions occur in the
context of geometric group actions.

\begin{lemma}
\label{Lemma: G-action implies unif. contr. and finite mdim}If a group $G$
acts geometrically on a contractible proper metric space $X$, then $X$ is
uniformly contractible and has finite macroscopic dimension.
\end{lemma}

\begin{proof}
Let $x_{0}\in X$, and choose $T>0$ sufficiently large that $GB\left(
x_{0},T\right)  =X$. Applying properness to $B[x_{0},T]$ shows that $\left\{
gB\left(  x_{0},T\right)  \mid g\in G\right\}  $ is a finite order open cover.

For uniform contractibility, let $R>0$ and $x\in X$ be arbitrary. Since
$B[x_{0},T+R]$ is compact, the contraction of $X$ restricts to a contraction
of $B[x_{0},T+R]$ in $B\left(  x_{0},S\right)  $ for some $S>0$. Choose $g\in
G$ such that $d\left(  x,gx_{0}\right)  <T$. Then $B\left(  x,R\right)
\subseteq B\left(  gx_{0},R+T\right)  $. Since the latter contracts in
$B\left(  gx_{0},S\right)  $, which is contained in $B\left(  x,T+S\right)  $,
then $B\left(  x,R\right)  $ contracts in $B\left(  x,T+S\right)  .$
\end{proof}

\section{Continuous approximations and proper homotopy
equivalences\label{Section: Continuous approximations}}

A crucial ingredient in this paper is an ability to approximate certain
functions with continuous ones. In this section, we develop the necessary
tools. Theorem \ref{Theorem: q-i groups act on phe spaces} and Corollary
\ref{Corollary: Cor to q-i groups act on phe spaces} will be almost immediate
consequences.\medskip

Let $\mathcal{U}$ be a locally finite open cover of a space $X$. The
\emph{nerve} of $\mathcal{U}$ is the abstract simplicial complex $N\left(
\mathcal{U}\right)  $ with vertex set $\mathcal{U}$ and a $k$-simplex
$\left\{  U_{0},U_{1},\cdots,U_{k}\right\}  $ whenever $\cap_{i=0}^{k}%
U_{i}\neq\varnothing$. Clearly $N\left(  \mathcal{U}\right)  $ is a locally
finite complex. When discussing the geometric realization $\left\vert N\left(
\mathcal{U}\right)  \right\vert $ we will denote the vertex corresponding to
$U\in\mathcal{U}$ by $v_{U}$. There is a partition of unity $\left\{
\lambda_{U}\right\}  _{U\in\mathcal{U}}$ where, for each $U_{0}\in\mathcal{U}%
$, $\lambda_{U_{0}}:X\rightarrow\left[  0,1\right]  $ is defined by%
\[
\lambda_{U_{0}}\left(  x\right)  =d\left(  x,X\backslash U_{0}\right)  /%
{\textstyle\sum\limits_{U\in\mathcal{U}}}
d\left(  x,X\backslash U\right)
\]
The local finiteness assumption ensures that the sums are finite and
continuous. Use these functions to define the barycentric map $\beta
:X\rightarrow\left\vert N\left(  \mathcal{U}\right)  \right\vert $ by%
\[
f\left(  x\right)  =%
{\textstyle\sum\limits_{U\in\mathcal{U}}}
\lambda_{U}\left(  x\right)  v_{U}%
\]
In other words, $x$ is taken to the point in the geometric realization of the
simplex $\left\{  U_{0},U_{1},\cdots,U_{k}\right\}  $ of all open sets
containing $x$ with barycentric coordinates $\lambda_{U_{i}}\left(  x\right)
$.

Our primary use of the above construction is the following.

\begin{lemma}
Suppose $X$ admits a uniformly bounded open cover of order $n+1$. Then $X$
contains a collection $\mathcal{A}^{0},\mathcal{A}^{1},\cdots,\mathcal{A}^{n}$
of locally discrete families of disjoint, uniformly bounded, closed sets which
together cover $X$.
\end{lemma}

\begin{proof}
By a theorem of general topology (see Lemma 41.6 in \cite{Mun00}) we may
assume that $\mathcal{U}$ is locally finite. Let $\beta:X\rightarrow\left\vert
N\left(  \mathcal{U}\right)  \right\vert $ be the barycentric map. Let
$N^{\prime}$ and $N^{\prime\prime}$ be the first and second derived
subdivision of $\left\vert N\left(  \mathcal{U}\right)  \right\vert $; then
for each $i$, let $\mathcal{A}^{i}$ be the collection of preimages of closed
star neighborhoods in $N^{\prime\prime}$ of the vertices of $N^{\prime}$ which
are barycenters of the $i$-simplices of $\left\vert N\left(  \mathcal{U}%
\right)  \right\vert $.
\end{proof}

\begin{proposition}
\label{Prop: continuous approximation}Let $f:X\rightarrow Y$ be a large-scale
uniform function, where $X$ has finite macroscopic dimension and $Y$ is a
uniformly contractible ANR. Then there is a continuous function
$g:X\rightarrow Y$ that is a bounded distance from $f$. If $E\subseteq X$ is a
closed set on which $f$ is already continuous, then $g$ may be chosen so that
$\left.  g\right\vert _{E}=\left.  f\right\vert _{E}$
\end{proposition}

\begin{proof}
Let $\mathcal{A}^{0},\mathcal{A}^{1},\cdots,\mathcal{A}^{n}$ be a finite set
of locally discrete collections of $K$-bounded closed sets which together
cover $X$, and let $\mathcal{A}$ denote the collection $\cup\mathcal{A}^{i}$.
From each $A\in\mathcal{A}$, choose a point $p_{A}$, and note that the set
$\left\{  p_{A}\right\}  $ is discrete. By adjoining these points to $E$ we
may assume, without loss of generality, that $E$ intersects each
$A\in\mathcal{A}$.

Define closed sets $C_{-1}\subseteq C_{0}\subseteq C_{1}\subseteq
\cdots\subseteq C_{n}$ by $C_{-1}=E$ and $C_{i}=C_{i-1}\cup\left(  \cup
_{A\in\mathcal{A}^{i}}A\right)  $ for each $i\geq0$. We will construct
$g:X\rightarrow Y$ inductively over the $C_{i}$.

Choose $L>0$ such that $d\left(  f\left(  x\right)  ,f\left(  x^{\prime
}\right)  \right)  <L$ whenever $d\left(  x,x^{\prime}\right)  <K$. Then let
$g_{-1}\equiv\left.  f\right\vert _{E}:C_{-1}\rightarrow Y$, and assume
inductively that there exists $L_{i}>0$ and a continuous function $g_{i}%
:C_{i}\rightarrow Y$ which is $2L_{i}$-close to $\left.  f\right\vert _{C_{i}%
}$ and agrees with $f$ on $E$.

To extend $g_{i}$ to $g_{i+1}:C_{i+1}\rightarrow Y$, let $A\in\mathcal{A}%
^{i+1}$. Since $f\left(  A\right)  \subseteq B\left(  f(p_{A}),L\right)  $,
then $g^{i}\left(  C^{i}\cap A\right)  \subseteq B\left(  f(p_{A}%
),L+2L_{i}\right)  $. Choose $L_{i+1}>L+2L_{i}$ so that each $B\left(
y,L+2L_{i}\right)  \subseteq Y$ contracts in $B\left(  y,L_{i+1}\right)  $. By
Corollaries \ref{Corollary: null-homotopic extension} and
\ref{Corollary: open subsets of ANRs are ANRs}, there is a continuous
extension of $\left.  g_{i}\right\vert _{A\cap C^{i}}$ to $g_{i+1}%
^{A}:A\rightarrow B\left(  y,L_{i+1}\right)  $; this map is necessarily
$2L_{i+1}$-close to $\left.  f\right\vert _{A}$. Take the union of $g_{i}$
with the $g_{i+1}^{A}$, over all $A\in\mathcal{A}^{i+1}$, to obtain a
continuous map $g_{i+1}:C_{i+1}\rightarrow Y$ which is $2L_{i+1}$-close to
$\left.  f\right\vert _{C_{i+1}}$ and is identical to $f$ on $E$. The
Proposition follows by induction.
\end{proof}

\begin{corollary}
\label{Corollary: boundedly close = boundedly homotopic}Suppose
$f,g:X\rightarrow Y$ are continuous, boundedly close, large-scale uniform
maps, where $X$ has finite macroscopic dimension and $Y$ is a uniformly
contractible ANR. Then $f$ and $g$ are boundedly homotopic.
\end{corollary}

\begin{proof}
Apply Proposition \ref{Prop: continuous approximation} to the situation
$J_{0}=f$ and $J_{t}=g$ for $0<t\leq1$ and $E=X\times\left\{  0,1\right\}  $
to get a continuous approximation $K:X\times\left[  0,1\right]  \rightarrow Y$.
\end{proof}

\begin{corollary}
\label{Corollary: coarse equivalence implies p.h.e.}Let $f^{\prime
}:X\rightarrow Y$ be a coarse equivalence between uniformly contractible
proper metric ANRs, each having finite macroscopic dimension. Then

\begin{enumerate}
\item $f^{\prime}$ is boundedly close to a continuous coarse equivalence
$f:X\rightarrow Y$,

\item $f$ (and hence $f^{\prime}$) has a continuous coarse inverse
$g:Y\rightarrow X$,

\item $gf$ and $fg$ are boundedly (hence properly) homotopic to
$\operatorname*{id}_{X}$ and $\operatorname*{id}_{Y}$, so $f$ and $g$ are
proper homotopy equivalences.\medskip
\end{enumerate}
\end{corollary}

The final observation of this section provides strong versions Theorem
\ref{Theorem: q-i groups act on phe spaces} and Corollary
\ref{Corollary: Cor to q-i groups act on phe spaces}.

\begin{theorem}
\label{Theorem: Strong version of Theorem 1}Suppose quasi-isometric groups $G$
and $H$ act geometrically on proper metric ANRs $X$ and $Y$, respectively.
Then $X$ and $Y$ are proper homotopy equivalent via continuous coarse equivalences.
\end{theorem}

\begin{proof}
Use Proposition \ref{Proposition: Generalized Svarc-Milnor} to conclude that
$X$ and $Y$ are coarsely equivalent to their respective groups and, thus, to
each other; then apply Lemma
\ref{Lemma: G-action implies unif. contr. and finite mdim} and the previous corollary.
\end{proof}

\section{$\mathcal{Z}$-Boundaries of
groups\label{Section: Z-boundaries of groups}}

In \cite{Bes96}, Bestvina introduced the notion of \textquotedblleft%
$\mathcal{Z}$-boundary of a group\textquotedblright, to provide a framework
that includes Gromov boundaries of word hyperbolic groups and visual
boundaries of CAT(0) groups, and can also be applied to other types of groups,
as well. To avoid some technical issues, he restricted attention to groups
that act properly, freely, and cocompactly (i.e., by covering transformations)
on finite-dimensional ARs (i.e., ERs). In \cite{Dra06}, Dranishnikov modified
the definition to allow for non-free actions and arbitrary (proper metric)
ARs; but the added flexibility came with a loss of generality in some key
theorems. Some of the lost generality was restored in \cite{Mor16a}; most of
the rest is taken care of in this paper. So, with recent progress taken into
account, Dranishnikov's version seems to be the \textquotedblleft
right\textquotedblright\ definition. It is:

\begin{definition}
A\emph{ }$\mathcal{Z}$\emph{-structure }on a group $G$ is a pair of spaces
$(\overline{X},Z)$ satisfying the following four conditions:

\begin{enumerate}
\item $\overline{X}$ is a compact AR,

\item $Z$ is a $\mathcal{Z}$-set in $\overline{X}$,

\item $X=\overline{X}-Z$ is a proper metric space on which $G$ acts
geometrically, and

\item $\overline{X}$ satisfies the following \emph{nullity condition} with
respect to the $G$-action on $X$: for every compact $C\subseteq X$ and any
open cover $\mathcal{U}$ of $\overline{X}$, all but finitely many $G$
translates of $C$ lie in an element of $\mathcal{U}$.\smallskip
\end{enumerate}

\noindent When this definition is satisfied, $Z$ is called a $\mathcal{Z}%
$\emph{-boundary }for $G$. If, in addition to the above, the $G$-action on $X$
extends to $\overline{X}$, the result is called an $E\mathcal{Z}%
$\emph{-structure} (equivariant $\mathcal{Z}$-structure), and $Z$ is called an
$E\mathcal{Z}$\emph{-boundary}\textbf{ }for $G$.\medskip
\end{definition}

\noindent\textbf{Examples:}

\begin{enumerate}
\item If $G$ acts geometrically on a proper CAT(0) space $X$, then
$\overline{X}=X\cup\partial_{\infty}X$, with the cone topology, gives an
$E\mathcal{Z}$-structure for $G$.

\item In \cite{BeMe91} it is shown that if $G$ is a hyperbolic group,
$P_{\rho}(G)$ is an appropriately chosen Rips complex, and $\partial G$ is the
Gromov boundary, then $\overline{P}_{\rho}(G)=P_{\rho}(G)\cup\partial G$
(appropriately topologized) gives an $E\mathcal{Z}$-structure for $G$.

\item Osajda and Przytycki \cite{OsPr09} have shown that systolic groups admit
$E\mathcal{Z}$-structures.

\item Guilbault, Moran, and Tirel \cite{GMT} have shown that Baumslag-Solitar
groups admit $E\mathcal{Z}$-structures.
\end{enumerate}

More general classes of groups have been addressed by Tirel \cite{Tir11} (free
and direct products), Dahmani \cite{Dah03} (relatively hyperbolic groups), and
Martin \cite{Mar14} (complexes of groups).\medskip

\begin{remark}
Bestvina's definition of $\mathcal{Z}$-structure did not require $G$ to act by
isometries on $X$, but only cocompactly by covering transformations. By the
following proposition, deduced from \cite{AMN11}, there is no loss of
generality in requiring a geometric action.
\end{remark}

\begin{proposition}
Suppose $G$ acts properly and cocompactly on a locally compact space $X$. Then
there is an equivalent proper metric for $X$ under which which the action is
by isometries.
\end{proposition}

Due to our use of coarse equivalences and the Generalized \v{S}varc-Milnor
Theorem, the following proposition (when it applies) is stronger than
necessary. We include it because it is interesting and not widely known.

\begin{proposition}
Suppose $G$ acts cocompactly by covering transformations on a connected,
locally connected space $X$. Then there is an equivalent proper geodesic
metric for $X$ under which the action is by isometries.
\end{proposition}

\begin{proof}
Let $p:X\rightarrow G\backslash X$ be the resulting covering projection. By
\cite{Bin52}, there is a geodesic metric $d^{\prime}$ on $G\backslash X$ which
generates the desired quotient topology. Lift that metric to $X$ by defining
\[
\rho\left(  x,y\right)  =\inf\left\{  \operatorname*{length}\left(
p\alpha\right)  \mid\alpha\text{ is a path in }X\text{ from }x\text{ to
}y\right\}
\]
Since $p$ is a local homeomorphism, $\rho$ generates the original topology on
$X$. By similar reasoning, $\left(  X,\rho\right)  $ is complete and locally
compact; so by the Hopf-Rinow Theorem \cite[Ch.I.3]{BrHa99}, $\rho$ is a
proper geodesic metric. Clearly the $G$-action on $\left(  X,\rho\right)  $ is
by isometries.
\end{proof}

Without the benefit of covering space theory, we do not know the answer to the
following:\medskip

\noindent\textbf{Question. }\emph{If a group }$G$\emph{ acts geometrically on
a proper metric AR }$\left(  X,d\right)  $\emph{ does there exist an
equivalent geodesic metric }$\rho$\emph{ under which the action is by
isometries?}\medskip

We close this section with an easy, but useful, lemma linking controlled
$\mathcal{Z}$-com\-pact\-ifi\-ca\-tions and $\mathcal{Z}$-structures.

\begin{lemma}
\label{Lemma: Z-structures are controlled Z-compactifications} If a group $G$
admits a $\mathcal{Z}$-structure $(\overline{X},Z)$, then $\overline{X}$ is a
controlled $\mathcal{Z}$-com\-pact\-i\-fi\-ca\-tion of $X=\overline{X}-Z$.
Conversely, if $G$ acts geometrically on a proper metric AR $\left(
X,d\right)  $ which admits a controlled $\mathcal{Z}$%
-com\-pact\-i\-fi\-ca\-tion $\overline{X}=X\cup Z$, then $(\overline{X},Z)$ is
a $\mathcal{Z}$-structure on $G$.
\end{lemma}

\begin{proof}
The initial observation can be found in \cite{GuMo15}. For the converse,
choose any compact set $C\subset X$ and any open cover $\mathcal{U}$ of
$\overline{X}$. Set $R=\operatorname*{diam}C+1$. From the control condition on
$\overline{X}$, there is a compact $K\subseteq X$ so that each subset of $X-K$
of diameter less than $R$ is contained in element of $\mathcal{U}$. Since the
action is by isometries $\operatorname*{diam}(gC)<R$ for all $g\in G$, and
since the action is proper only finitely many translates of $C$ intersect $K$;
thus, all but finitely many lie in an element of $\mathcal{U}$.
\end{proof}

\section{Boundary swapping\label{Section: Boundary swapping}}

We now obtain proofs of Theorems \ref{Theorem: Boundary Swapping Theorem} and
\ref{Theorem: Generalized Boundary Swapping}. Both follow from a more general
theorem that does not involve group actions. Our approach expands upon one
suggested by Ferry \cite{Fer00}.

\begin{theorem}
\label{Theorem: general boundary swapping theorem}Let $f:\left(
X,d_{X}\right)  \rightarrow\left(  Y,d_{Y}\right)  $ be a coarse equivalence
between uniformly contractible proper metric spaces, each having finite
macroscopic dimension, and suppose $Y$ admits a controlled $\mathcal{Z}%
$-com\-pact\-i\-fi\-ca\-tion $\overline{Y}=Y\cup Z$. Then $X$ admits a
controlled $\mathcal{Z}$-com\-pact\-i\-fi\-ca\-tion $\overline{X}=X\cup Z$. If
$f$ is continuous, this may be done so that $f$ extends to a continuous map
$\overline{f}:\overline{X}\rightarrow\overline{Y}$ which is the identity on
$Z$.
\end{theorem}

\begin{proof}
If necessary, use Corollary 5.4 to replace $f$ with a continuous coarse
equivalence and choose a continuous coarse inverse $g:Y\rightarrow X$ and a
bounded homotopy $J:X\times\left[  0,1\right]  \rightarrow X$ with
$J_{0}=\operatorname*{id}_{X}$ and $J_{1}=gf$.

Extend $f$ to a function $\overline{f}:X\sqcup Z\rightarrow\overline{Y}$ by
letting $\overline{f}$ be the identity on $Z$. Then give $X\sqcup Z$ the
topology generated by the open subsets of $X$ and sets of the form
$\overline{f}^{-1}(U)$ where $U\subset\overline{Y}$ is open, and let
$\overline{X}$ denote the resulting topological space. Clearly, $\overline
{f}:\overline{X}\rightarrow\overline{Y}$ is continuous and $\overline{X}$ is
compact, Hausdorff, and second countable. It follows that $\overline{X}$ is
metrizable and separable. The theorem will be proved by showing that
$\overline{X}$ is a controlled $\mathcal{Z}$-com\-pact\-i\-fi\-ca\-tion of $X$.

Before proceeding, we establish some notation. Let $\overline{g}:\overline
{Y}\rightarrow\overline{X}$ and $\overline{J}:\overline{X}\times\left[
0,1\right]  \rightarrow\overline{X}$ be the obvious extensions which are the
identity on $Z$, and fix metrics $\overline{d}_{X}$ and $\overline{d}_{Y}$ to
$\overline{X}$ and $\overline{Y}$, respectively (these are \emph{not
}extensions of $d_{X}$ and $d_{Y})$. Whenever $\overline{U}$ denotes a subset
of $\overline{X}$ [resp., $\overline{Y}$], $U$ will denote $\overline{U}\cap
X$ [resp., $\overline{U}\cap Y$]. Finally, select $K>0$ so that $J$ is a
$K$-homotopy and $fg$ is $K$-close to $\operatorname*{id}_{Y}$.\medskip

\noindent\textbf{Claim 1.}\emph{ }$\overline{X}$ \emph{is a }$\mathcal{Z}%
$\emph{-com\-pact\-i\-fi\-ca\-tion of }$X$\emph{.}\medskip

This claim follows from Lemma \ref{Lemma: Z-set criterion}, provided that
$\overline{g}$ and $\overline{J}$ are continuous (all other hypotheses are immediate).

To see that $\overline{g}$ is continuous at $z\in Z$, let $\overline{f}%
^{-1}\left(  \overline{U}\right)  $ be a basic open neighborhood of
$\overline{g}\left(  z\right)  =z$ in $\overline{X}$. Then $z\in\overline{U}$
and by Lemma \ref{Lemma: local property of controlled Z-compactifications},
there is a smaller open neighborhood $\overline{V}$ of $z$ in $\overline{Y}$
such that $d_{Y}\left(  V,Y-U\right)  >K$. If $y\in V$ then $d_{Y}\left(
y,f\left(  g\left(  y\right)  \right)  \right)  \leq K$, so $f\left(  g\left(
y\right)  \right)  \in U\subseteq\overline{U}$; therefore $g\left(  y\right)
\in\overline{f}^{-1}\left(  \overline{U}\right)  $. It follows that
$\overline{g}\left(  \overline{V}\right)  \subseteq\overline{f}^{-1}\left(
\overline{U}\right)  $, so $\overline{g}$ is continuous at $z$.

To prove continuity of $\overline{J}$, we will need an analog of Lemma
\ref{Lemma: local property of controlled Z-compactifications} that can be
applied to $\overline{X}$.\medskip

\noindent\textbf{Subclaim.}\emph{ Given }$z\in Z$\emph{, a neighborhood
}$\overline{U}$\emph{ of }$z$\emph{ in }$\overline{X}$\emph{, and }%
$R>0$\emph{, there is a neighborhood }$\overline{V}$\emph{ of }$z$\emph{ so
that }$d(V,X-U)\geq R$\emph{.}\medskip

\emph{Proof of subclaim. }Since $f$ is a coarse map, choose $S>0$ so that
whenever $d_{X}(x,x^{\prime})<R$ then $d_{Y}(f(x),f(x^{\prime}))<S$. Without
loss of generality, we may assume that $\overline{U}=\overline{f}%
^{-1}(\overline{W})$, where $\overline{W}$ is open in $\overline{Y}$. By Lemma
\ref{Lemma: local property of controlled Z-compactifications}, $z$ has an open
neighborhood $\overline{W^{\prime}}\subset\overline{W}$ $d_{Y}(W^{\prime
},Y-W)>S$. Let $\overline{V}=\overline{f}^{-1}(\overline{W^{\prime}})$ and
note that if $x\in V$, $y\in X-U$, and $d(x,y)<R$; then $f(x)\in W^{\prime}$,
$f(y)\in Y-W$, and $d_{Y}(f(x),f(y))\leq S$, a contradiction which proves the
subclaim.\medskip

To see that $\overline{J}$ is continuous at $\left(  z,t\right)  \in
Z\times\left[  0,1\right]  $, let $\overline{U}$ be a neighborhood of
$\overline{J}\left(  z,t\right)  =z$ in $\overline{X}$. Apply the subclaim to
find a neighborhood $\overline{V}$ of $z$ in $\overline{X}$ such that
$d_{X}\left(  V,X-U\right)  >K$. For any $x\in V$, $J\left(  x\times\left[
0,1\right]  \right)  $ intersects $V$ and has diameter $\leq K$, so $J\left(
V\times\left[  0,1\right]  \right)  \subseteq U$. It follows that
$\overline{J}\left(  \overline{V}\times\left[  0,1\right]  \right)
\subseteq\overline{U}$, so $\overline{J}$ is continuous at $\left(
x,t\right)  $ and Claim 1 is complete.\medskip

\noindent\textbf{Claim 2.}\emph{ }$\overline{X}$\emph{ is a controlled
}$\mathcal{Z}$\emph{-com\-pact\-i\-fi\-ca\-tion of }$\left(  X,d_{X}\right)
$\emph{.}\medskip

Fix $R>0$ and let $\mathcal{U}$ be a cover of $\overline{X}$ by basic open
subsets. In particular, the elements of $\mathcal{U}$ that intersect $Z$ are
of the form $\overline{f}^{-1}\left(  \overline{W}\right)  $ where
$\overline{W}$ is an open subset of $\overline{Y}$ intersecting $Z$. Let
\[
\mathcal{W}=\{\overline{W}\subseteq\overline{Y}\mid\overline{f}^{-1}%
(\overline{W})\in\mathcal{U}\text{ and }\overline{W}\cap Z\neq\varnothing\}
\]
Since $\mathcal{W}$ covers $Z$, $C\equiv\overline{Y}-\cup\mathcal{W}$ is a
compact subset of $Y$. Let $D$ be a compact subset of $Y$ such that
$C\subseteq\operatorname*{int}D$. Then $\mathcal{W}^{\ast}=\mathcal{W\cup
}\left\{  \operatorname*{int}D\right\}  $ is an open cover of $\overline{Y}$.
Choose $S>0$ so that, if $A\subseteq X$ and $\operatorname*{diam}_{d_{X}%
}\left(  A\right)  \leq R$, then $\operatorname*{diam}_{d_{Y}}\left(
f(A)\right)  \leq S$. Then choose compact $E\subseteq Y$ such that $E\supseteq
D$ and subsets of $Y-E$ of $d_{Y}$-diameter $<S$ lie in an element of
$\mathcal{W}^{\ast}$. Then $f^{-1}\left(  E\right)  \subseteq X$ is compact
and if $A\subseteq X-f^{-1}\left(  E\right)  $ with $\operatorname*{diam}%
_{d_{X}}\left(  A\right)  <R$, then $f\left(  A\right)  $ lies in an element
of $\mathcal{W}^{\ast}$ which is clearly not $\operatorname*{int}D$. So
$f\left(  A\right)  \subseteq\overline{W}$ for some $\overline{W}%
\in\mathcal{W}$, and therefore $A\subseteq\overline{f}^{-1}(\overline{W}%
)\in\mathcal{U}$.
\end{proof}

\begin{corollary}
[Generalized $\mathcal{Z}$-boundary Swapping Theorem]If quasi-isometric groups
$G$ and $H$ act geometrically on proper metric ARs $X$ and $Y$, resp., and $Y$
can be compactified to a $\mathcal{Z}$-structure $\left(  \overline
{Y},Z\right)  $ for $H$, then $X$ can be compactified by addition of the same
boundary to a $\mathcal{Z}$-structure $\left(  \overline{X},Z\right)  $ for
$G$.
\end{corollary}

\begin{proof}
By Lemma \ref{Lemma: G-action implies unif. contr. and finite mdim} both $X$
and $Y$ are uniformly contractible with finite macroscopic dimension, and by
Proposition \ref{Proposition: Generalized Svarc-Milnor} they are coarsely
equivalent. Since $\overline{Y}$ is a controlled $\mathcal{Z}$%
-com\-pact\-i\-fi\-ca\-tion of $Y$ (Lemma
\ref{Lemma: Z-structures are controlled Z-compactifications}), Theorem
\ref{Theorem: general boundary swapping theorem} provides a corresponding
controlled $\mathcal{Z}$-com\-pact\-i\-fi\-ca\-tion of $X$. Another
application of Lemma
\ref{Lemma: Z-structures are controlled Z-compactifications} assures the
desired $\mathcal{Z}$-structure on $G$.
\end{proof}

The non-equivariant version of standard boundary swapping is now immediate.

\begin{corollary}
[$\mathcal{Z}$-boundary Swapping Theorem]%
\label{Corollary: Non-equivariant boundary swapping}If $G$ acts geometrically
on proper metric ARs $X$ and $Y$, and $Y$ can be compactified to a
$\mathcal{Z}$-structure $\left(  \overline{Y},Z\right)  $ for $G$, then $X$
can be compactified by addition of the same boundary to another $\mathcal{Z}%
$-structure $\left(  \overline{X},Z\right)  $ for $G$.
\end{corollary}

The $E\mathcal{Z}$-version of Corollary
\ref{Corollary: Non-equivariant boundary swapping} requires some additional
work. We already have a $\mathcal{Z}$-com\-pact\-i\-fi\-ca\-tion $\overline
{X}=X\sqcup Z$; moreover, $G$ acts on $X$ and $Z$, individually (the latter by
restricting the action on $\overline{Y}$). The idea is to combine these into a
single $G$-action on $\overline{X}$. For each $\gamma\in G$, if $\gamma
_{X}:X\rightarrow X$ and $\gamma_{Z}:Z\rightarrow Z$ are the corresponding
homeomorphisms under the sub-actions, we need $\gamma_{X}\cup\gamma
_{Z}:\overline{X}\rightarrow\overline{X}$ to be a homeomorphism. For that, it
is enough to verify continuity at points of $Z$.

A variation on the continuity argument for $\overline{g}$, used in Theorem
\ref{Theorem: general boundary swapping theorem}, proves:

\begin{lemma}
\label{Lemma: boundedly close maps extend}Let $\overline{X}=X\sqcup Z$ be a
controlled $\mathcal{Z}$-com\-pact\-i\-fi\-ca\-tion of a proper metric space
$X$ and suppose $f,f^{\prime}:X\rightarrow X$ are boundedly close continuous
functions. If $\overline{f}:\overline{X}\rightarrow\overline{X}$ is a
continuous extension of $f$ that takes $Z$ into $Z$, then $\overline
{f^{\prime}}=f^{\prime}\cup\left.  \overline{f}\right\vert _{Z}$ is a
continuous extension of $f^{\prime}$.
\end{lemma}

In Theorem \ref{Theorem: general boundary swapping theorem}, we began with a
continuous coarse equivalence $f:X\rightarrow Y$ and a continuous a coarse
inverse $g:Y\rightarrow X$. Those were extended to continuous functions
$\overline{f}:\overline{X}\rightarrow\overline{Y}$ and $\overline{g}%
:\overline{Y}\rightarrow\overline{X}$, both of which are the identity on $Z$.
If, in addition, the $G$-action on $Y$ extends to $\overline{Y}$, then for
every $\gamma\in G$, $\overline{g}\gamma\overline{f}:\overline{X}%
\rightarrow\overline{X}$ is a continuous extension of $g\gamma f$ which agrees
with $\gamma$ on $Z$. By Lemma \ref{Lemma: boundedly close maps extend}, if
$g\gamma f$ is boundedly close to $\gamma$ on $X$, we can extend $\gamma$ to
$\overline{X}$ using $\left.  \gamma\right\vert _{Z}$. So $E\mathcal{Z}%
$-boundary swapping is completed by the following proposition.

\begin{proposition}
Suppose proper metric ARs $X$ and $Y$ each admit geometric $G$-actions. Then
there exists a continuous coarse equivalence $f:X\rightarrow Y$, a continuous
coarse inverse $h:Y\rightarrow X$, and a constant $K>0$ such that, for every
$\gamma\in G$, $h\gamma f$ is $K$-close to $\gamma$.
\end{proposition}

\begin{proof}
Fix points $x_{0}\in X$ and $y_{0}\in Y$. Let $\alpha:G\rightarrow X$ be
defined as $\alpha(g)=gx_{0}$ and $\beta:G\rightarrow Y$ be given by
$\beta(g)=gy_{0}$. Since $G$ acts geometrically on $X$ and $Y$, Proposition
4.5 guarantees, when $G$ is endowed with a word metric, that $\alpha$ and
$\beta$ are coarse equivalences. Define a coarse inverse $\alpha^{\prime
}:X\rightarrow G$ as follows: for each $gx_{0}\in Gx_{0}$, let $\alpha
^{\prime}(gx_{0})=g^{\prime}$ where $g^{\prime}\in G$ is such that
$gx_{0}=g^{\prime}x_{0}$. Having defined $\alpha^{\prime}$ on the orbit of
$x_{0}$, we may now define $\alpha^{\prime}$ on all of $X$. Since $G$ acts
cocompactly on $X$, there exists an $R>0$ such that $X=GB(x_{0},R)$. Then, for
$x\in X$, choose a $g\in G$ such that $x\in gB(x_{0},R)$ and let
$\alpha^{\prime}(x)=\alpha^{\prime}(gx_{0})$. Define a coarse inverse
$\beta^{\prime}:Y\rightarrow G$ in a similar fashion.

Now, let $f=\beta\alpha^{\prime}$ and $h=\alpha\beta^{\prime}$. As these are
compositions of coarse equivalences, both $f$ and $h$ are coarse equivalences.
We may further assume they are continuous from Corollary 5.4. We will show
these are the desired coarse equivalences.

Since $G$ acts geometrically on $X$ and $Y$, stabilizers of points have
uniformly bounded diameters (e.g., \cite[I.8.5]{BrHa99}). Choose $M>0$ such
that $\hbox{diam}_{G}G_{z}<M$ for all $z\in X\cup Y$. Then choose $L>0$ so
that whenever $d_{G}(g_{1},g_{2})<M$, $d_{X}(\alpha(g_{1}),\alpha
(g_{2}))=d_{X}(g_{1}x_{0},g_{2}x_{0})<L$, and set $K=R+2L$. We will show that
$h\gamma f$ is $K$-close to $\gamma$ for all $\gamma\in G$.

Let $x\in X$ and $\gamma\in G$. Choose $g\in G$ so that $x\in gB(x_{0},R)$ and
note that $d_{X}(\gamma x,\gamma gx_{0})<R$. Now, consider, $h\gamma
f(x)=\alpha\beta^{\prime}\gamma\beta\alpha^{\prime}(x)$. By definition,
$\alpha^{\prime}(x)=\alpha^{\prime}(gx_{0})=g_{1}$ for some $g_{1}\in G$ with
$g_{1}x_{0}=gx_{0}$. Notice that since $g^{-1}g_{1}\in G_{x_{0}}$, then
$d_{G}(g,g_{1})<M$ and thus $d_{X}(gx_{0},g_{1}x_{0})<L$. Now, $\beta^{\prime
}\gamma\beta\alpha^{\prime}(x)=\beta^{\prime}\gamma\beta(g_{1})=\beta^{\prime
}\gamma(g_{1}y_{0})=\beta^{\prime}(\gamma g_{1}y_{0})=g_{2}$ where $g_{2}\in
G$ is such that $\gamma g_{1}y_{0}=g_{2}y_{0}$. Since $g_{2}^{-1}g_{1}\in
G_{y_{0}}$, $d_{G}(\gamma g_{1},g_{2})<M$ and thus $d_{X}(\gamma g_{1}%
x_{0},g_{2}x_{0})<L$.

Thus, we have $h\gamma f(x)=\alpha(g_{2})=g_{2}x_{0}$ and we observe:
\[
d_{X}(\gamma x,h\gamma f(x))=d_{X}(\gamma x,g_{2}x_{0})
\]%
\[
\leq d_{X}(\gamma x,\gamma gx_{0})+d_{X}(\gamma gx_{0},\gamma g_{1}%
x_{0})+d_{X}(\gamma g_{1}x_{0},g_{2}x_{0})<R+L+L=K.
\]

\end{proof}

\section{Shape equivalence of $\mathcal{Z}$-boundaries}

We now prove Theorem \ref{Theorem: shape of bdry a quasi-isom invariant} and
Corollary \ref{Corollary: Shape of bdry a group invariant}. For readers
familiar with shape theory, these are almost immediate consequences of the
Boundary Swapping Theorems. That connection and the remaining technical
details were worked out in \cite[\S 3.7]{Gui16}; but there the conclusions
were less general since a full-blown boundary swapping theorem was not yet
known. Here we complete the picture. We will begin with a brief review of
shape theory and then outline the main argument. Additional details can be
found in \cite[\S 3.7]{Gui16}. For comprehensive treatments of shape theory,
see \cite{Bor71}, \cite{DySe78} or \cite{MaSe82}.\medskip

To define the \emph{shape} of a compact metric space $A$, one first associates
to $A$ an inverse sequence of finite polyhedra and continuous maps
\[
K_{0}\overset{f_{1}}{\longleftarrow}K_{1}\overset{f_{2}}{\longleftarrow}%
K_{2}\overset{f_{3}}{\longleftarrow}\cdots
\]
Such a sequence can be obtained in a variety of ways. For example: the $K_{i}$
can be chosen to be nerves of progressively finer finite open covers of $A$;
or, if $A$ can be embedded in $%
\mathbb{R}
^{n}$, the $K_{i}$ can be progressively smaller polyhedral neighborhoods of
$A$ connected by inclusion maps. Clearly, these sequences are not uniquely
determined by $A$. A pair of inverse sequences of finite polyhedra $\left\{
K_{i},f_{i}\right\}  $ and $\left\{  L_{i},g_{i}\right\}  $ are
\emph{pro-homotopy equivalent }if they contain subsequences that fit into a
ladder diagram%
\[
\begin{diagram} K_{i_{0}} & & \lTo^{f_{i_{0},i_{1}}} & & K_{i_{1}} & & \lTo^{f_{i_{1},i_{2}}} & & K_{i_{2}} & & \lTo^{f_{i_{2},i_{3}}}& & K_{i_{3}}& \cdots\\ & \luTo & & \ldTo & & \luTo & & \ldTo & & \luTo & & \ldTo &\\ & & L_{j_{0}} & & \lTo^{g_{j_{0},j_{1}}} & & L_{j_{1}} & & \lTo^{g_{j_{1},j_{2}}}& & L_{j_{2}} & & \lTo^{g_{j_{2},j_{3}}} & & \cdots \end{diagram}
\]
in which each triangle of maps homotopy commutes. (Doubly subscripted maps are
compositions.) It can be shown that any two inverse sequences associated with
$A$ are pro-homotopy equivalent. The pro-homotopy class of these sequences
determines the \emph{shape} of $A$, i.e., compact metric spaces $A$ and
$A^{\prime}$ are defined to be \emph{shape equivalent} if their associated
inverse sequences are pro-homotopy equivalent.

If $A$ is a compact subset of an ANR $X$, one can always choose a nested
sequence of compact neighborhoods%
\begin{equation}
L_{0}\hookleftarrow L_{1}\hookleftarrow L_{2}\hookleftarrow\cdots
\label{Sequence: neighborhood system}%
\end{equation}
with $\cap L_{i}=A$. If, in addition, the $L_{i}$ themselves can be chosen to
be ANRs (as is the case for the most commonly studied $X$), Theorem
\ref{Theorem: West's Theorem} allows us to use sequence
(\ref{Theorem: Boundary Swapping Theorem}) to represent the shape of $A$. From
there, a proof of Theorem \ref{Theorem: q-i groups act on phe spaces} is
rather easy. For the general case, we will use just a bit of the theory of
Hilbert cube manifolds (in particular Theorem \ref{Theorem: Edwards' Theorem})
along with the boundary swapping techniques developed here, to replace $X$
with a space for which that easy strategy can be applied.

\begin{remark}
For those who prefer to avoid Hilbert cube manifolds, a more general (but
slightly more technical) approach to shape theory can be used to circumvent
their need. That approach, developed in \cite{MaSe82}, allows the use of
infinite CW complexes (or noncompact ANRs) in associated inverse sequences
representing a compactum $A$. That means an allowable sequence of type
(\ref{Sequence: neighborhood system}) is immediately available---just choose
the $L_{i}$ to be \emph{open} neighborhoods of $A$. (This approach was used by
Ontaneda for similar purposes in \cite{Ont05}.) Whichever approach one uses,
boundary swapping helps to make the final conclusion straightforward, as the
following argument shows.
\end{remark}

\begin{proof}
[Proof of Theorem \ref{Theorem: shape of bdry a quasi-isom invariant}]Suppose
quasi-isometric groups $G$ and $H$ admit $\mathcal{Z}$-structures with
boundaries $Z$ and $Z^{\prime}$, respectively. Then, by Theorem
\ref{Theorem: general boundary swapping theorem}, there exists a proper metric
AR $X$ and a pair of $\mathcal{Z}$-com\-pact\-i\-fi\-ca\-tions $\overline
{X}=X\cup Z$ and $\overline{X}^{\prime}=X\cup Z$. Assume for the moment that
there is a sequence
\[
\overline{N}_{0}\hookleftarrow\overline{N}_{1}\hookleftarrow\overline{N}%
_{2}\hookleftarrow\cdots
\]
of compact ANR neighborhoods of $Z$ in $\overline{X}$ with $\cap\overline
{N}_{i}=Z$. By Corollary \ref{Corollary: open subsets of ANRs are ANRs} and
Lemma \ref{Lemma: Z-compactifications are ANRs}, this is equivalent to the
existence of a cofinal sequence
\[
N_{0}\hookleftarrow N_{1}\hookleftarrow N_{2}\hookleftarrow\cdots
\]
of closed ANR neighborhoods of infinity in $X$, where by Lemma
\ref{Lemma: Z-sets in subspaces}, each $\overline{N}_{i}$ is a $\mathcal{Z}%
$-com\-pact\-i\-fi\-ca\-tion of $N_{i}$. Since each $N_{i}\hookrightarrow
\overline{N}_{i}$ is a homotopy equivalence, choose homotopy inverses
$g_{i}:\overline{N}_{i}\rightarrow N_{i}$.

By the boundary swap, each $N_{i}$ has an alternative $\mathcal{Z}%
$-com\-pact\-i\-fi\-ca\-tion $\overline{N}_{i}^{\prime}=N_{i}\cup Z^{\prime
}\subseteq\overline{X}^{\prime}$, giving rise to a sequence
\[
\overline{N}_{0}^{\prime}\hookleftarrow\overline{N}_{1}^{\prime}%
\hookleftarrow\overline{N}_{2}^{\prime}\hookleftarrow\cdots
\]
of compact ANR neighborhoods of $Z^{\prime}$ with $\cap\overline{N}%
_{i}^{\prime}=Z^{\prime}$. For each $i$, let $g_{i}^{\prime}:\overline{N}%
_{i}^{\prime}\rightarrow N_{i}$ be a homotopy inverse for $N_{i}%
\hookrightarrow\overline{N}_{i}^{\prime}$.

Begin with trivial ladder diagram%
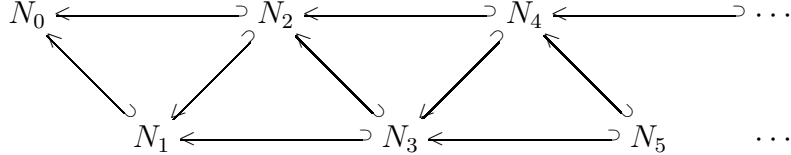
\[
\begin{diagram} N_{0} & & \lInto & & N_{2} & & \lInto & & N_{4} & & \lInto & & \cdots \\ & \luInto & & \ldInto & & \luInto & & \ldInto & & \luInto \\ & & N_{1} & & \lInto & & N_{3} & & \lInto & & N_{5} & & \cdots \end{diagram}
\]
then compactify the top row by adding copies of $Z$ and the bottom by adding
copies of $Z^{\prime}$. Use maps $\operatorname*{incl}\circ g_{i}^{\prime}$
and $\operatorname*{incl}\circ g_{i}$ for up and down arrows to obtain a
homotopy commuting diagram%
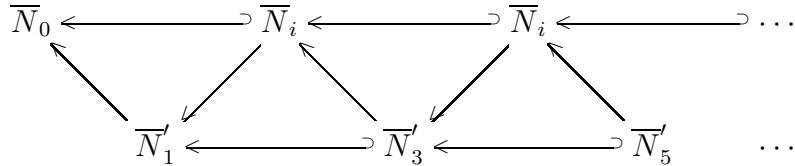
\[
\begin{diagram} \overline{N}_{0} & & \lInto & & \overline{N}_{i} & & \lInto & & \overline{N}_{i} & & \lInto & & \cdots \\ & \luTo & & \ldTo & & \luTo & & \ldTo & & \luTo \\ & & \overline{N}_{1}^{\prime} & & \lInto & & \overline{N}_{3}^{\prime} & & \lInto & & \overline{N}_{5}^{\prime} & & \cdots \end{diagram}
\]
which proves that $Z$ and $Z^{\prime}$ are shape equivalent.

To complete the proof, we must address the situation where $Z$ does not have
arbitrarily small compact ANR neighborhoods in $\overline{X}$ (equivalently,
$X$ does not contain arbitrarily small closed ANR neighborhoods of infinity).
In that case, we will replace $X$ with the AR $X\times\left[  0,1\right]
^{\infty}$. By Theorem \ref{Theorem: Generalized Boundary Swapping} or
Corollary \ref{Corollary: Non-equivariant boundary swapping}, we can
$\mathcal{Z}$-compactify $X\times\left[  0,1\right]  ^{\infty}$ with either
$Z$ or $Z^{\prime}$, and by Theorem \ref{Theorem: Edwards' Theorem},
$X\times\left[  0,1\right]  ^{\infty}$ is a Hilbert cube manifold. By standard
Hilbert cube manifold topology \cite{Cha76}, $X\times\left[  0,1\right]
^{\infty}$ contains arbitrarily small closed Hilbert cube manifold
neighborhoods of infinity. Since these are ANRs, the general case follows from
the earlier special case.
\end{proof}

\bibliographystyle{amsalpha}
\bibliography{Biblio}

\end{document}